\documentclass[12pt,reqno]{amsart}

\usepackage{amssymb, amsmath, amsthm}
\usepackage[colorlinks=true, urlcolor=blue, linkcolor=blue, citecolor=blue]{hyperref}
\usepackage[alphabetic,lite,nobysame]{amsrefs}
\usepackage{verbatim}
\usepackage{amscd}   
\usepackage[all]{xy} 
\usepackage{youngtab} 
\usepackage{young} 
\usepackage{ytableau}
\usepackage{tikz}
\usepackage{ mathrsfs }
\usepackage{cases}
\usepackage{array}
\usepackage{cellspace}
\usepackage{tabu}
\usepackage{calligra,mathrsfs}
\usepackage{bm}
\usepackage{mathtools}
\usepackage{dotlessi}
\usepackage[margin=1in]{geometry}

\newcommand{\defi}[1]{{\bf \upshape\sffamily #1}}
\DeclareMathOperator{\ShHom}{\mathscr{H}\text{\kern -3pt {\calligra\large om}}\,}

\renewcommand{\a}{\alpha}
\renewcommand{\b}{\beta}
\newcommand{\bw}{\bigwedge}

\newcommand{\CC}{\mathbb{C}}

\renewcommand{\ll}{\lambda}

\newcommand{\onto}{\twoheadrightarrow}
\newcommand{\oo}{\otimes}
\newcommand{\pd}{\partial}

\newcommand{\GL}{\operatorname{GL}}

\newcommand{\rk}{\operatorname{rank}}
\newcommand{\SL}{\operatorname{SL}}

\newcommand{\Sym}{\operatorname{Sym}}
\newcommand{\Tor}{\operatorname{Tor}}

\newcommand{\coker}{\operatorname{coker}}
\renewcommand{\det}{\operatorname{det}}

\renewcommand{\ker}{\operatorname{ker}}

\newcommand{\reg}{\operatorname{reg}}

\newcommand{\bb}[1]{\mathbb{#1}}

\newcommand{\mc}[1]{\mathcal{#1}}
\newcommand{\mf}[1]{\mathfrak{#1}}

\newcommand{\op}[1]{\operatorname{#1}}

\def\PP{{\mathbf P}}
\def\lra{\longrightarrow}
\def\lla{\longleftarrow}

\newtheorem{theorem}[equation]{Theorem}
\newtheorem*{theorem*}{Theorem}
\newtheorem*{problem*}{Problem}
\newtheorem{lemma}[equation]{Lemma}
\newtheorem{conjecture}[equation]{Conjecture}
\newtheorem{proposition}[equation]{Proposition}

\newtheorem*{corollary*}{Corollary}

\theoremstyle{definition}

\newtheorem*{definition*}{Definition}

\theoremstyle{remark}
\newtheorem{remark}[equation]{Remark}
\newtheorem*{remark*}{Remark}

\numberwithin{equation}{section}

\begin{document}

\title[Powers of binary forms and derived Hermite reciprocity]{Powers of binary forms and \\derived Hermite reciprocity}

\author{Claudiu Raicu}
\address{Department of Mathematics, University of Notre Dame, 255 Hurley, Notre Dame, IN 46556\newline
\indent Institute of Mathematics ``Simion Stoilow'' of the Romanian Academy}
\email{craicu@nd.edu}

\author{Steven V Sam}
\address{Department of Mathematics, University of California, San Diego, La Jolla, CA 92093}
\email{ssam@ucsd.edu}

\author{Jerzy Weyman}
\address{Department of Mathematics, Jagiellonian University, Krak\'ow, Poland}
\email{jerzy.weyman@gmail.com}

\author{Fuxiang Yang}
\address{Department of Mathematics, University of Notre Dame, 255 Hurley, Notre Dame, IN 46556}
\email{fyang6@nd.edu}


\date{September 1, 2026}


\begin{abstract} 
 For $a,b\geq 1$, Hilbert found in 1886 a collection of polynomial equations that cut out set-theoretically the variety $X$ parametrizing $a$-th powers of binary forms of degree~$b$. We determine the ideal of all polynomials vanishing on $X$, showing that it is generated in degree $b+1$ and that it has a linear minimal free resolution. We do this by generalizing results of Abdesselam and Chipalkatti on an analogue of the Foulkes--Howe map and by establishing a derived analogue of the classical Hermite reciprocity theorem for complexes of $\SL_2$-representations. In our investigation, we are led to the ideal generated by the subrepresentation $\Sym^{ab}(\CC^2) \subset \Sym^a(\Sym^b \CC^2)$. We determine its Castelnuovo--Mumford regularity in general and the minimal free resolution for small values of $b$. 
\end{abstract}

\maketitle

\section{Introduction}\label{sec:intro}

The study of binary forms was central to the 19th century theory of invariants and covariants, going back to Cayley and Sylvester, and continuing in the work of Hermite, Clebsch, Gordan, and Hilbert, but even today many basic questions regarding equations and syzygies for varieties of binary forms remain open. We let $\PP^d = \bb{P}(\Sym^d\CC^2)$ denote the projective space parametrizing complex binary forms of degree $d$ up to scaling. For any partition $\ll\vdash d$, we get a subvariety (a \defi{factorization locus}, or \defi{coincidence root locus})
\[ 
X_{\ll} = \{[F]\in\PP^d : F=L_1^{\ll_1}L_2^{\ll_2}\cdots,\text{ where each $L_i$ is a linear form}\}.
\]
When $\ll=(d)$, $X_{\ll}$ is the rational normal curve of degree $d$, while for $\ll=(1^d)$ we get $X_{\ll}=\PP^d$. Other cases of interest occur for $\ll=(2,1^{d-2})$, when $X_{\ll}$ is the discriminant hypersurface, or $\ll=(d-1,1)$, when $X_{\ll}$ is the tangent developable surface to the rational normal curve. It is a fundamental question, which remains largely open, to understand the minimal generators of the defining ideal $I(X_{\ll})$, as well as the higher syzygy modules. 

Beyond a handful of examples in small dimensions and easy cases $\ll=(2,1^{d-2})$, $\ll=(1^d)$, the complete description of the syzygies is only known for $\ll=(d)$ \cite{DE-syzygies}*{Section~6A}, $\ll=(d-1,1)$ \cite{AFPRW}*{Theorem~1.6}, and $\ll=(d/2,d/2)$ with $d$ even (\cite{abd-chi-BG-loci}*{Theorem~1.4} and \cite{abd-chi-bipartite}*{Theorem~3}). The main goal of this paper is to address the case when $\ll=(a^b)$ has equal parts, where $d=ab$. Hilbert constructed in \cite{hilbert} covariant equations that cut out $X_{(a^b)}$ set-theoretically (see also \cites{abd-chi-HC,ottaviani} for a historical account), and our main result resolves the corresponding ideal-theoretic problem, and confirms \cite{abd-chi-HC}*{Conjecture~5.1}.

\begin{theorem}\label{thm:Xab}
 Suppose $a,b\geq 2$, set $d=ab$, and $X=X_{(a^b)}$. The homogeneous ideal $I(X)$ is generated by polynomials of degree $b+1$ (the maximal minors of a matrix of linear forms), and its minimal free resolution is linear. The projective dimension of $I(X)$ is $(d-1)$.
\end{theorem}

We have excluded from Theorem~\ref{thm:Xab} the trivial cases $a=1$ (when $X=\PP^d$) and $b=1$ (when $X$ is a rational normal curve), which are \defi{arithmetically Cohen--Macaulay (ACM)}--i.e., their homogeneous coordinate ring satisfies the Cohen--Macaulay property. All cases covered by Theorem~\ref{thm:Xab} are not ACM, and in particular the linear resolution of $I(X)$ does not come from an Eagon--Northcott complex. The matrix of linear forms giving the determinantal structure for $X$ is highly non-generic, and we explain its construction following~\cite{abd-chi-HC}.

\begin{subequations}
We write $U=\bb{C}^2 = \bb{C}x_1\oplus\bb{C}x_2$, and consider the bilinear map
\begin{equation}\label{eq:bilinear-Jacobian}
 \omega \colon \Sym^d U \times \Sym^b U \lra \Sym^{d+b-2}U,\quad \omega(F,G) = \det\begin{bmatrix} F_{x_1} & F_{x_2} \\ G_{x_1} & G_{x_2}\end{bmatrix},
\end{equation}
where $H_{x_i}=\partial H / \partial x_i$ denotes the partial derivative of $H$ with respect to $x_i$. For $d=ab$, it follows from \cite{abd-chi-HC}*{Proposition~3.1} that for nonzero $F\in\Sym^d U$  and $G\in\Sym^b U$, we have
\begin{equation}\label{eq:Jac-criterion}
\omega(F,G) = 0 \Longleftrightarrow [F]=[G^a]\in\PP^d.
\end{equation}
We can now view $\omega$ as a family of linear maps parametrized by $F\in\Sym^d U$, and as such it gives rise to a sheaf morphism
\begin{equation}\label{eq:def-omega-onPd}
 \omega\colon \Sym^b U \oo \mc{O}_{\PP^d}(-1) \lra \Sym^{d+b-2} U \oo \mc{O}_{\PP^d},
\end{equation}
which by \eqref{eq:Jac-criterion} drops rank precisely along $X$. 
\end{subequations}
We prove in Lemma~\ref{lem:beta-surj} that the $(b+1)$-minors of \eqref{eq:def-omega-onPd} generate $I(X)$, and in Proposition~\ref{prop:explicit-betti} we determine all the algebraic Betti numbers of $X$. In particular, the number of minimal generators of $I(X)$ is given by
\[\dim_{\bb{C}} I(X)_{b+1} = {d+b+1\choose b+1}-{d+a+b\choose b}.\]
In the simplest case $a=b=2$ we get that $I(X)$ is generated by the $3$-minors of a $5\times 3$ matrix of linear forms. The Eagon--Northcott complex would then yield $10$ cubic generators that satisfy $9$ linear syzygies, but in fact the Betti table of $I(X)$ is
\[
\begin{array}{c|cccc}
       &0&1&2&3 \\ \hline
     3&7&10&5&1\\
\end{array}
\]
so that the linear map \eqref{eq:def-omega-onPd} is quite degenerate.

\begin{subequations}
Part of Theorem~\ref{thm:Xab} is the assertion that $I(X)$ contains no equations of degree $\leq b$. To understand this, note that $X$ can be viewed as the image of a morphism (which is in fact a closed immersion)
\begin{equation}\label{eq:param-nua-forX}
 \nu_a \colon \PP^b \lra \PP^{d},\quad [G] \mapsto [G^a].
\end{equation}
This parametrization \eqref{eq:param-nua-forX} induces for each $k$ a restriction map
\begin{equation}\label{eq:def-alfak}
\xymatrix{
H^0\left(\PP^d,\mc{O}_{\PP^d}(k)\right) \ar[r]^-{\alpha_k} \ar@{=}[d] & H^0\left(\PP^b,\mc{O}_{\PP^b}(ak)\right) \ar@{=}[d]\\  
\Sym^k(\Sym^{ab} U) \ar[r]^-{\alpha_k} & \Sym^{ak} (\Sym^b U).
}
\end{equation}
Following the historical remarks of \cite{abd-chi-BG-loci}, we will refer to $\alpha_k$ as the \defi{Foulkes--Howe map}.
\end{subequations}

\begin{theorem}\label{thm:Foulkes-Howe-map}
 The Foulkes--Howe map \eqref{eq:def-alfak} is injective for $k\leq b$ and surjective for $k\geq b$.
\end{theorem}

This resolves \cite[Conjecture 1]{AC} (see also \cite{bergeron} for related problems). Theorem~\ref{thm:Foulkes-Howe-map} reduces to proving that for $k=b$ the map
\begin{equation}\label{eq:weird-Hermite}
\alpha_b\colon \Sym^b(\Sym^d U) \lra \Sym^d(\Sym^b U)
\end{equation}
is an isomorphism. It is well-known that the source and target of $\alpha_b$ are abstractly isomorphic as $\GL_2$-representations, a fact usually referred to as Hermite reciprocity \cite{hermite}. Hermite provided an explicit isomorphism, which has many incarnations illustrated in \cite{rai-sam}, but the map $\alpha_b$ is not the Hermite map (see Section~\ref{sec:differentHermite}). In the special case $b=2$, Theorem~\ref{thm:Foulkes-Howe-map} is essentially equivalent to \cite{abd-chi-BG-loci}*{Theorem~1.1}, whose proof is quite elaborate.

Our approach to Theorem~\ref{thm:Foulkes-Howe-map} is cohomological and is illustrated next in a special case. We identify throughout $U^{\vee}=U$ via the perfect pairing $U\times U\lra\bw^2 U \cong \CC$. We can then use \eqref{eq:bilinear-Jacobian} to construct a sheaf map
\begin{equation}\label{eq:def-phi}
 \varphi \colon  \Sym^{d+b-2}U\oo \mc{O}_{\PP^b}(-1) \lra \Sym^d U \oo \mc{O}_{\PP^b},
\end{equation}
whose cokernel is the line bundle $\mc{O}_{\PP^b}(a)$ (see Section~\ref{subsec:presen-Oa}). We can view \eqref{eq:def-phi} as a $2$-term complex concentrated in cohomological degrees $-1,0$, and form its $b$-fold symmetric power $\mc{S}_b^{\bullet}$, where
\[ 
\mc{S}_b^{-i} = \bw^i(\Sym^{d+b-2}U)\oo\Sym^{b-i}(\Sym^d U)\oo \mc{O}_{\PP^b}(-i)\quad\text{ for }i=0,\dots,b.
\]
The degree $0$ cohomology sheaf of $\mc{S}_b^{\bullet}$ is $\mc{H}^0\left(\mc{S}_b^{\bullet}\right) = \mc{O}_{\PP^b}(d)$ and the induced map
\[ H^0\left(\PP^b,\mc{S}_b^0\right) \lra H^0\left(\PP^b,\mc{O}_{\PP^b}(d)\right)\]
is the map \eqref{eq:weird-Hermite}. The main technical difficulty in our work is to prove that this induced map is an isomorphism, but the case $b=2$ is deceptively simple and we explain it here. We augment the complex $\mc{S}_b^{\bullet}$ with the natural map $\mc{S}_b^0\lra\mc{H}^0\left(\mc{S}_b^{\bullet}\right)$, and get a complex $\widehat{\mc{S}_b^{\bullet}}$ whose cohomology sheaves are concentrated in negative degrees. For $b=2$ it takes the form
\[0 \to \bw^2(\Sym^d U)(-2) \to \Sym^d U \oo \Sym^d U(-1) \to \Sym^2(\Sym^d U) \oo \mc{O}_{\PP^2} \to \mc{O}_{\PP^2}(d) \to 0,\]
and since $\mc{O}_{\PP^2}(-2)$ and $\mc{O}_{\PP^2}(-1)$ are acyclic, the interesting hypercohomology groups of $\widehat{\mc{S}_2^{\bullet}}$ are computed by
\[ \bb{H}^0\left(\widehat{\mc{S}_2^{\bullet}}\right) = \ker(\alpha_2) \quad\text{and}\quad \bb{H}^1\left(\widehat{\mc{S}_2^{\bullet}}\right) = \coker(\alpha_2) .\]
Now the map \eqref{eq:def-phi} has kernel $\mc{O}_{\PP^2}(-a-1)$, hence the complex $\widehat{\mc{S}_2^{\bullet}}$ has a unique non-zero cohomology sheaf, namely $\mc{H}^{-1}\left(\widehat{\mc{S}_2^{\bullet}}\right) = \mc{O}_{\PP^2}(-1)$. Since $\mc{O}_{\PP^2}(-1)$ is acyclic, this implies the hypercohomology of $\widehat{\mc{S}_2^{\bullet}}$ vanishes identically, hence $\alpha_2$ is an isomorphism.

For a general $b$, the same argument shows that \eqref{eq:weird-Hermite} is an isomorphism if and only if the hypercohomology groups of $\widehat{\mc{S}_b^{\bullet}}$ vanish, but establishing this is more delicate. It requires establishing an isomorphism of complexes
\[
\Sym^{b-1}(\varphi) \cong \bigwedge^{b-1}(\varphi^\vee(-1))
\] 
which we refer to as \defi{derived Hermite reciprocity} since term-by-term this can be viewed as a variant of Hermite reciprocity, see Section~\ref{sec:char-comp} (however, compatibility with the differential is the most difficult and important point). This can be rephrased as saying that the $(b-1)$-fold symmetric power complex $\mc{S}_{b-1}^{\bullet}$ is isomorphic to its dual (up to a cohomological shift, and a twist by $\mc{O}_{\PP^b}(b-1)$, see Theorem~\ref{thm:self-duality}). When $b=2$, this is because the map \eqref{eq:def-phi} is self-adjoint, but in general it is more subtle, and it follows from a more general duality theorem that we prove for constant rank morphisms between vector bundles, see Section~\ref{sec:general-dual}. For a similar self-duality statement in the context of Hermite reciprocity see \cite{rai-sam}*{Section~5}, but unlike in loc. cit., the complex $\mc{S}_{b-1}^{\bullet}$ is no longer a resolution (see Section~\ref{sec:self-duality}).

The map \eqref{eq:param-nua-forX} is defined by the linear series $\Sym^d U \subset H^0\left(\PP^b,\mc{O}_{\PP^b}(a)\right)$, where the inclusion comes from the unique copy of $\Sym^d U$ inside the plethysm $\Sym^a(\Sym^b U)$, and we can view \eqref{eq:def-phi} as describing the (linear) syzygies between the sections in $\Sym^d U$. We write $S=\Sym(\Sym^b U)$, and let $I_{a,b} = \langle \Sym^d U\rangle$ denote the corresponding primary ideal.

\begin{theorem}\label{thm:reg-Iab}
 The Castelnuovo--Mumford regularity of $I_{a,b}$ is 
 \[ \op{reg}(I_{a,b}) = \left\lfloor\frac{b+2}{2}\right\rfloor\cdot a-\left\lfloor\frac{b}{2}\right\rfloor.\] 
\end{theorem}

We would like to understand the minimal free resolution of $I_{a,b}$, especially as it connects via Hermite reciprocity with coincident root loci $X_\lambda$ where $\lambda = (p, 1^{b-p})$ is a hook partition (see Section 2 of \cite{weyman-mult} and in particular \cite{weyman-mult}*{Theorem~2}). More generally, a natural problem in elimination theory / implicitization is to understand powers of the base ideal $I_{a,b}$ and the corresponding Rees algebra. As a first step we make the following conjecture.
\begin{conjecture} \label{conj:reg-pow-conj}
For $1 \le j\le b$, we have
\begin{equation*}
\op{reg}(I_{a,b}^j) = \left\lfloor \frac{b+j+1}{2}\right\rfloor a-\left\lfloor \frac{b-j+1}{2}\right\rfloor,
\end{equation*}
and the minimal generators of $I_{a,b}^j$ form a Gr\"obner basis. 
\end{conjecture}

When $j=1$ this is proved in Theorem~\ref{thm:reg-Iab}, while for $j=b$ it is a consequence of the isomorphism \eqref{eq:weird-Hermite}, which shows that $I_{a,b}^b = \mf{m}^d$ for $\mf{m}$ the maximal homogeneous ideal in $S$. We verify the regularity formula for $j=b-1$ in Theorem~\ref{thm:res-Iab-b-1}, where we describe the full Betti table for $I_{a,b}^{b-1}$.

The resolution of $I_{a,b}$ is easy to describe for small values of $b$, but it is unknown in general. When $b=1$ we have that $I_{a,1}=\mf{m}^a$ is a power of the maximal ideal, which has a Hilbert--Burch resolution
\[ 0 \lla I_{a,1} \lla S(-a)^{a+1} \overset{\varphi}{\lla} S(-a-1)^a \lla 0\]
where $\varphi$ is the map of free modules induced by \eqref{eq:def-phi} with $b=1$ and $d=a$. Similarly, when $b=2$ we get a Buchsbaum--Eisenbud resolution
\[ 0 \lla I_{a,2} \lla S(-a)^{2a+1} \overset{\varphi}{\lla} S(-a-1)^{2a+1} \lla S(-2a-1) \lla 0\]
which follows from the special case $b=2$, $d=2a$ of \eqref{eq:def-phi}, when $\ker(\varphi)=\mc{O}_{\PP^2}(-a-1)$ as noted earlier. In the case $b=3$ we prove the following.

\begin{theorem}\label{thm:res-Ia3}
 The minimal free resolution for the ideals $I_{a,3}$ has the form
 \begin{align*} 0\lla I_{a,3} \lla S(-a)^{3a+1} \overset{\varphi}{\lla} S(-a-1)^{3a+2} \oplus S(-2a)^{a+1\choose 2} \qquad \\
 \lla S(-2a-1)^{a^2+a+2} \lla S(-2a-2)^{a+1\choose 2} \lla 0. 
 \end{align*}
\end{theorem}

\noindent Notice that this recovers Theorem~\ref{thm:reg-Iab} for $b=3$, since $\reg(I_{a,3})=2a-1$. For a specific example, consider the case when $a=2$, when the Betti table of $I_{2,3}$ takes the form:
\[
\begin{array}{c|ccccccc}
       &0&1&2&3 \\ \hline
     2&7&8&-&-\\
     3&-&3&8&3\\
\end{array}
\]
The linear part of the first differential is defined again by \eqref{eq:def-phi}, and we will explain how the second linear strand arises from a monad for an $\SL_2$-invariant instanton bundle on~$\PP^3$, discussed in \cite{faenzi}. We will also make Theorem~\ref{thm:res-Ia3} more precise by exhibiting the syzygy modules as representations of $\SL_2$ (see Section~\ref{subsec:syzygies-Ia3}).

\subsection*{Acknowledgements}
 We thank Giorgio Ottaviani and Ethan Reed for helpful discussions. Experiments with the computer algebra software Macaulay2 \cite{GS} have provided numerous valuable insights. Raicu acknowledges the support of the National Science Foundation Grant DMS-2302341. Sam acknowledges the support of the National Science Foundation Grant DMS-2302149. Weyman was supported by the grants MAESTRO NCN-UMO-2019/34/A/ST1/00263 Research in Commutative Algebra and Representation Theory and OPUS 28 NCN-2024/55/B/ST1/01437 - The structure of perfect ideals and linkage. Yang acknowledges the support of the Arthur J. Schmitt Fellowship and the National Science Foundation Grant DMS-2302341.

\section{Preliminaries}\label{sec:prelim}

\subsection{Notation and conventions} 

Given a vector space $V$, $\bb{P}(V)$ denotes the projective space parametrizing 1-dimensional quotients of $V$. Throughout the paper we let $U=\CC^2$, we choose an identification $\bw^2 U = \bb{C}$, and use it to identify $U^{\vee}=U$ via the perfect pairing $U\times U\lra\bw^2 U=\CC$. We write $\PP^d = \bb{P}(\Sym^d U)$ for the projective space parametrizing degree $d$ binary forms up to scaling. All the maps we write down in relation to binary forms will be equivariant for the natural action of the special linear group $\SL(U)=\SL_2$. 

Given a sheaf $\mc{F}$ on a space $X$ we let $H^i(\mc{F})=H^i(X,\mc{F})$ denote its cohomology groups. Given a complex of sheaves $\mc{F}^{\bullet}$, we write 
\[
\mc{H}^i(\mc{F}^{\bullet}) = \frac{\ker(\mc{F}^i \to \mc{F}^{i+1})}{\mathrm{im}(\mc{F}^{i-1} \to \mc{F}^i)}
\]
for the $i$-th cohomology sheaf of the complex, and $\bb{H}^i(\mc{F}^{\bullet})=\bb{H}^i(X,\mc{F}^{\bullet})$ for the $i$-th hypercohomology group of $\mc{F}^{\bullet}$.

\subsection{Maximal minors and hypercohomology}\label{subsec:max-minors}

Consider finite dimensional vector spaces $V_1,V_0,W$, and a linear map
\[\gamma\colon  V_1 \lra V_0 \oo W.\]
Suppose that $\dim(W)=r+1$, so that $\bb{P}(W)\cong\PP^r$, and consider the map of sheaves
\[ V_1 \oo \mc{O}_{\PP^r}(-1) \lra V_0 \oo \mc{O}_{\PP^r}\]
induced by $\gamma$. Thinking of this map as a $2$-term cochain complex concentrated in degrees $-1$ and $0$, we can form its $(r+1)$-fold symmetric power $\Gamma^{\bullet}$, where
\[ 
\Gamma^{-i} = \bw^i V_1 \oo \Sym^{r+1-i}V_0 \oo \mc{O}_{\PP^r}(-i) \quad\text{ for }i=0,\dots,r+1.
\]
Note that all of the cohomology groups of  $\Gamma^{-i}$ vanish for $i=1,\dots,r$. Moreover, the only nonvanishing cohomology groups for $\Gamma^0$ and (if we assume that $\dim(V_1)\geq r+1$) $\Gamma^{-r-1}$ are:
\[ 
H^0(\PP^r,\Gamma^0) = \Sym^{r+1}V_0,\quad H^r(\PP^r,\Gamma^{r+1}) = \bw^{r+1}V_1 \oo \bw^{r+1}W^{\vee}.
\]
It follows from the hypercohomology spectral sequence for $\Gamma^{\bullet}$ that there exists a map (on page $r$ of the spectral sequence)
\begin{equation}\label{eq:map-dr}
 d_r \colon  \bw^{r+1}V_1 \oo \bw^{r+1}W^{\vee} \lra \Sym^{r+1}V_0,
\end{equation}
and we have
\[ \bb{H}^0(\Gamma^{\bullet}) = \coker(d_r),\quad \bb{H}^{-1}(\Gamma^{\bullet}) = \ker(d_r),\quad \bb{H}^i(\Gamma^{\bullet}) = 0\text{ for }i\neq -1,0.\]
To interpret the map $d_r$ concretely, we choose bases for $V_1,V_0,W$, and think of the basis of $V_0$ as the variables $x_1,x_2,\dots$ in a polynomial ring $\Sym(V_0)$. If $\dim(V_1) = m$ then we can interpret the map $\gamma$ as an $m\times (r+1)$ matrix of linear forms in $V_0$. 

\begin{proposition}
The image of $d_r$ is spanned by the $(r+1)\times(r+1)$ minors of $\gamma$. 
\end{proposition}

\begin{proof} 
Consider the universal example where $V_0 = V_1 \otimes W^\vee$ and $\gamma$ is the $\GL(V_1) \times \GL(W)$-invariant map coming from the evaluation $W^\vee \otimes W \to \CC$. Then $d_r$ is nonzero and is an equivariant map, and the representation $\bigwedge^{r+1} V_1 \otimes \bigwedge^{r+1} W^\vee$ appears with multiplicity $1$ in $\Sym^{r+1}(V_1 \otimes W^\vee)$, so the image of $d_r$ is the stated space of minors. Finally, we can relate general examples to this universal example by applying a base change.
\end{proof}

\subsection{Graded modules associated to coherent sheaves} \label{subsec:gradedmod}

We let $S$ denote the homogeneous coordinate ring of $\PP^b$, and to a coherent sheaf $\mc{E}$ on~$\PP^b$ we associate the graded $S$-modules
\[
H^i_{*}(\mc{E}) \coloneqq \bigoplus_{j\in\bb{Z}} H^i\left(\mc{E}\oo\mc{O}_{\PP^b}(j)\right),\quad\text{ for }i=0,\dots,b.
\]
One has that $H^0_{*}(\mc{E})$ is a free module when $\mc{E}$ is a direct sum of line bundles. The following lemma will be used in the proofs of Theorems~\ref{thm:res-Ia3} and~\ref{thm:res-Iab-b-1}.

\begin{lemma}\label{lem:induced-H0}
    Suppose that we have an exact complex of sheaves on $\PP^b$
    \[
    \mc{C}^\bullet \colon\quad 0 \lra \mc{C}^{-b} \lra \cdots \lra \mc{C}^{-1} \lra \mc{C}^0 \lra \mc{C}^1 \lra 0,
    \]
    where each term $\mc{C}^i$ is a direct sum of line bundles on $\PP^b$. The complex of graded $S$-modules
    \[
    H^0_{*} (\mc{C}^\bullet) \colon\quad 0 \lra H^0_{*}(\mc{C}^{-b}) \lra \cdots \lra H^0_{*}(\mc{C}^{-1}) \lra H^0_{*}(\mc{C}^0) \overset{\pd}{\lra} H^0_{*}(\mc{C}^{1}) 
    \]
    is acyclic, i.e., its cohomology is $0$ in degrees $\le 0$.
\end{lemma}

\begin{proof}
 Since $\mc{C}^{\bullet}$ is exact, the same is true about all twists $\mc{C}^\bullet(j)$, and therefore the corresponding hypercohomology groups vanish. We get a spectral sequence
    \[ 
    E^{-p,q}_1 = H^q_*(\mc{C}^{-p}) \quad\Longrightarrow\quad \bb{H}^{q-p}_*(\mc{C}^{\bullet}) = 0.
    \]
    Since each $\mc{C}^i$ is a direct sum of line bundles, $E^{-p,q}_1$ vanishes for $q\neq 0,b$. It follows that beyond the first page, the only potentially non-trivial map in the spectral sequence is
    \[ d_{b+1}^{-b,b}\colon E_{b+1}^{-b,b} \to E_{b+1}^{1,0}.\]
    It follows that the only non-zero cohomology of $E_1^{\bullet,0}$ may occur in cohomological degree $1$ (and must agree with the cohomology of $E_1^{\bullet,b}$ in cohomological degree $-b$), as desired.
\end{proof}

\section{A self-duality theorem}\label{sec:self-duality} 

\subsection{A general duality} \label{sec:general-dual}

Fix a non-negative integer $r$. We consider maps of vector bundles
\[
  \mathcal{A}^{-1} \overset{d}{\lra} \mathcal{A}^0
\]
on a locally ringed space $X$ such that:
\begin{enumerate}
\item $d$ has constant rank,
  
\item $\mathcal{E}_\mathcal{A}:=\ker d$ has rank $r$, and
\item $\mathcal{L}_\mathcal{A}:=\coker d$ is a line bundle.
\end{enumerate}
Morphisms of such maps $f \colon \mathcal{A} \to \mathcal{B}$ are defined to be commutative squares. A morphism $f$ induces maps $\mathcal{E}_A \to \mathcal{E}_B$ and $\mathcal{L}_A \to \mathcal{L}_B$. Let $\theta_\mathcal{A} \colon \mathcal{A} \to \mathcal{L}_A$ be the natural surjection and define $\delta_\mathcal{A}$ by
\[
\delta_{\mathcal{A}} \colon \mathcal{A} \otimes \mathcal{A} \lra \mathcal{A} \otimes \mathcal{L}_A, \qquad \delta_{\mathcal{A}} = (1 \otimes \theta) - \tau \circ (\theta \otimes 1)
\]
where $\tau \colon \mathcal{L}_\mathcal{A} \otimes \mathcal{A} \to \mathcal{A} \otimes \mathcal{L}_{\mathcal{A}}$ is the symmetry map. Since $\mathcal{L}_\mathcal{A}$ is a line bundle, the composition
\[\mc{A} \otimes \mc{A} \overset{\delta_\mc{A}}{\lra} \mc{A} \otimes \mc{L}_{\mc{A}} \overset{\theta_\mc{A} \otimes 1}{\lra} \mc{L}_{\mc{A}} \otimes \mc{L}_{\mc{A}}\]
is the zero map. Hence, the image is contained in $\mathcal{K}_\mathcal{A} \otimes \mathcal{L}_\mathcal{A}$ where $\mathcal{K}_\mathcal{A} = \ker(\theta_\mathcal{A})$. Note that $\delta_\mathcal{A}$ is functorial. Indeed, given $f \colon \mathcal{A} \to \mathcal{B}$, we have a commutative square
\[
  \xymatrix{ \mathcal{A} \otimes \mathcal{A} \ar[d]_-{f \otimes f} \ar[r]^-{\delta_\mathcal{A}} & \mathcal{K}_\mathcal{A} \otimes \mathcal{L}_\mathcal{A} \ar[d]^-{f \otimes f} \\
    \mathcal{B} \otimes \mathcal{B} \ar[r]^-{\delta_\mathcal{B}} & \mathcal{K}_\mathcal{B} \otimes \mathcal{L}_\mathcal{B}}.
\]

Let $D^r$ denote the $r$th divided power functor. The inclusion $\mathcal{E}[1] \to \mathcal{K}$ is a quasi-isomorphism; we define $\psi_\mathcal{A}$ in the derived category to be the composition
\begin{equation}\label{eq:psi}
     D^r(\mathcal{A}) \otimes D^r(\mathcal{A}) \lra D^r(\mathcal{A} \otimes \mathcal{A}) \xrightarrow{D^r(\delta_\mathcal{A})} D^r(\mathcal{K}_A \otimes \mathcal{L}_A) \xleftarrow{\cong} D^r(\mathcal{E}_\mathcal{A}[1] \otimes \mathcal{L}_\mathcal{A}).
\end{equation}
Note that the rightmost term simplifies to $\det(\mc{E}_{\mc{A}})[r] \otimes \mc{L}_{\mc{A}}^{\otimes r}$.

\begin{theorem} \label{thm:general-duality}
  $\psi_\mathcal{A}$ is a perfect pairing, i.e., the map given by adjunction
  \[D^r(\mathcal{A}) \lra (D^r \mathcal{A})^\vee \otimes \det(\mc{E}_{\mc{A}})[r] \otimes \mc{L}^{\otimes r}_{\mc{A}}\]
  is an isomorphism in the derived category of $X$.
\end{theorem}

\begin{proof}
    If $\phi \colon Y \to X$ is a morphism of locally ringed spaces, then $\phi^*(\psi_\mathcal{A})$ is naturally identified with $\psi_{\phi^*(\mathcal{A})}$. In particular, $\psi$ is compatible with pullback, and we may immediately reduce to the case where $X$ is the spectrum of a local ring. 
    
    For fixed $X$, it follows from construction that $\psi$ is functorial in $\mc{A}$. In particular, if $f \colon \mathcal{A} \to \mathcal{B}$ is a quasi-isomorphism, then $\psi_\mathcal{A}$ is a perfect pairing if and only if $\psi_{\mathcal{B}}$ is. Since $X$ is the spectrum of a local ring, $\mathcal{A}$ splits as a direct sum of the form
    \[(\mathcal{E}_A \overset{0}{\lra} \mathcal{L}_A) \oplus (T \overset{1}{\lra} T)\]
    where $T$ is a free module. Hence, it suffices to consider the case where the differential $d$ is the zero map. In this case, $\mathcal{K}_A = \mathcal{E}_A[1]$ and we can omit the last quasi-isomorphism in~(\ref{eq:psi}). 
    
    We compute the restriction of $\psi_{\mc{A}}$ to the term $D^r(\mathcal{A})^{-i} \otimes D^r(\mathcal{A})^{-j}$. Since $D^r(\mc{K}_{\mc{A}} \otimes \mc{L}_{\mc{A}}) = D^r(\mc{E}_{\mc{A}}[1] \otimes \mc{L}_{\mc{A}})$ is concentrated in cohomological degree $-r$, the map $\psi$ is zero unless $i + j = r$. Moreover, $\mc{A} \otimes \mc{A}$ is given by
    \[\left(\mc{E}_{\mc{A}}[1] \otimes \mc{E}_{\mc{A}}[1]\right) \oplus \left((\mc{E}_{\mc{A}}[1] \otimes \mc{L}_{\mc{A}}) \oplus (\mc{L}_{\mc{A}} \otimes \mc{E}_{\mc{A}}[1]) \right) \oplus \left( \mc{L}_{\mc{A}}\otimes \mc{L}_{\mc{A}}\right).\]
    By the $(\GL(\mc{A}) \times \GL(\mc{A}))$-equivariance of the first map in (\ref{eq:psi}), $D^r(\mathcal{A})^{-i} \otimes D^r(\mathcal{A})^{-r+i}$ maps isomorphically to $D^i(\mc{E}_{\mc{A}}[1] \otimes \mc{L}_{\mc{A}}) \otimes D^{r-i}(\mc{L}_{\mc{A}} \otimes \mc{E}_{\mc{A}}[1])$. To conclude the proof, it suffices to show that the corresponding restriction of $D^r(\delta_\mc{A})$ given by
    \begin{equation}\label{eq:resD(delta)}
        D^i(\mc{E}_{\mc{A}}[1] \otimes \mc{L}_{\mc{A}}) \otimes D^{r-i}(\mc{L}_{\mc{A}} \otimes \mc{E}_{\mc{A}}[1]) \lra D^r(\mc{E}_{\mc{A}}[1]\otimes \mc{L}_{\mc{A}})=\det(\mc{E}_{\mc{A}})[r] \otimes \mc{L}^{\otimes r}_{\mc{A}}
    \end{equation}
    is a perfect pairing. Concretely, the map $\delta_{\mc{A}}$ is the zero map on $(\mc{A}\otimes \mc{A})^0$ and $(\mc{A}\otimes \mc{A})^{-2}$, and $\delta_{\mc{A}}(x,y) = x-\tau(y)$ on $(\mc{A}\otimes \mc{A})^{-1}$, so (\ref{eq:resD(delta)}) is given by
    \[x_1\cdots x_i \otimes y_1 \cdots y_{r-i} \longmapsto (-1)^{r-i}x_1 \cdots x_i \tau(y_1) \cdots \tau(y_{r-i})\]
    which is a perfect pairing.
\end{proof}

\subsection{Application to binary forms} \label{subsec:presen-Oa}

We fix $a,b\geq 1$, let $\PP^b = \bb{P}(\Sym^b U)$, and set $d=ab$.

By \cite{BFL}*{Theorem~2} there exists an $\SL_2$-equivariant exact sequence
\begin{equation}\label{eq:pres-Oa}
 0 \lra \mc{F} \lra \Sym^{d+b-2}U \oo \mc{O}_{\PP^b}(-1) \overset{\varphi}{\lra} \Sym^{d}U \oo \mc{O}_{\PP^b} \lra \mc{O}_{\PP^b}(a) \lra 0,
\end{equation}
where $\varphi$ is as in \eqref{eq:def-phi}. We have that
\[ \rk(\mc{F}) = b-1 \qquad\mbox{and}\qquad \det(\mc{F}) = \mc{O}_{\PP^b}(-d-b+a+1).\]

If we dualize $\varphi$ and twist by $\mc{O}_{\PP^b}(-1)$ to get (after the identification $U\cong U^{\vee}$)
\[ \varphi^{\vee}(-1) \colon  \Sym^{d}U(-1) \lra \Sym^{d+b-2} U \oo \mc{O}_{\PP^b}.\]
Thinking of $\varphi$ and $\varphi^{\vee}(-1)$ as $2$-term complexes concentrated in cohomological degrees $-1$ and $0$, we construct associated complexes by taking symmetric or exterior powers as in \cite{weyman}*{Section~2.4}.

For $r\geq 0$, we consider the symmetric power $\mc{S}_r^{\bullet} = \Sym^r(\varphi)$. It is a complex whose terms are
\[ 
\mc{S}_r^{-i} = \bw^i(\Sym^{d+b-2}U) \oo \Sym^{r-i}(\Sym^d U) \oo \mc{O}_{\PP^b}(-i),\text{ for }0\leq i\leq r,
\]
where the cohomology sheaves are given by (recall that $\mc{F} =\ker \varphi$)
\[
\mc{H}^{-i}(\mc{S}^{\bullet}_r) = \bigwedge^i\mc{F} \otimes \mc{O}_{\PP^b}((r-i)a).
\]
Similarly, we consider the exterior power $\mc{W}^r_{\bullet}= \bw^r(\varphi^{\vee}(-1))$, whose terms are
\[ 
\mc{W}_r^{-i} = \Sym^i(\Sym^d U) \oo \bw^{r-i}(\Sym^{d+b-2}U) \oo \mc{O}_{\PP^b}(-i),\text{ for }0\leq i\leq r,
\]
where the cohomology sheaves are given by
\[
\mc{H}^{-i}(\mc{W}^{\bullet}_r) = \bigwedge^{r-i} \mc{F}^\vee \otimes\mc{O}_{\PP^b}(-ia-r).
\]
Notice that the complexes $\mc{S}_r^{\bullet}$ and $\mc{W}_r^{\bullet}$ are dual to each other (up to a twist by some line bundle, and a cohomological shift). The complexes $\mc{S}_1^{\bullet}$ and $\mc{W}_1^{\bullet}$ are given by the maps $\varphi$ and $\varphi^{\vee}(-1)$, and we have $\mc{S}_r^{\bullet} = \Sym^r(\mc{S}_1^{\bullet})$ and $\mc{W}_r^{\bullet} = \bw^r(\mc{W}_1^{\bullet})$. When $r = b - 1$, the complexes $\mc{S}_{b-1}^\bullet$ and $\mc{W}_{b-1}^\bullet$ have isomorphic cohomology sheaves. Recall from Section~\ref{sec:prelim} that
\[ \rk(\mc{F}) = b-1 \qquad\mbox{and}\qquad \det(\mc{F}) = \mc{O}_{\PP^b}(-d-b+a+1).\]
It follows that
\begin{align*}
    \mc{H}^{-i}(\mc{S}^{\bullet}_{b-1}) = \bigwedge^i\mc{F} \otimes \mc{O}_{\PP^b}((b-1-i)a) &\cong \bigwedge^{b-1-i}\mc{F}^\vee \otimes \det(\mc{F}) \otimes \mc{O}_{\PP^b}((b-1-i)a)\\  
    &= \bigwedge^{b-1-i}\mc{F}^\vee \otimes \mc{O}_{\PP^b}(-ia-b+1) = \mc{H}^{-i}(\mc{W}^\bullet_{b-1}).
\end{align*}

When $b=2$ it is not hard to check that $\varphi$ is a skew-symmetric map, and in particular, we get an isomorphism between~$\mc{S}_1^{\bullet}$ and~$\mc{W}_1^{\bullet}$. In the case when $a=1$, it follows from \cite{reed}*{Theorem~1.1} that $\mc{S}_{b-1}^{\bullet} \cong \mc{W}_{b-1}^{\bullet}$. We generalize this for all $b\geq 1$ and all $d=ab$ as follows.

\begin{theorem}\label{thm:self-duality}
 We have a $\GL(U)$-equivariant isomorphism of complexes $\mc{S}_{b-1}^{\bullet} \cong \mc{W}_{b-1}^{\bullet}$.
\end{theorem}

\begin{proof}
It follows from Theorem~\ref{thm:general-duality} that $\mc{S}_{b-1}^{\bullet}\simeq \mc{W}_{b-1}^{\bullet}$ in the derived category $D(\PP^b)$. 

We claim that the complexes $H^0_\ast(\mc{S}_{b-1}^{\bullet})$ and $H^0_\ast(\mc{W}_{b-1}^{\bullet})$ are homotopy equivalent. Write $S=\Sym(\Sym^b U)$ for the homogeneous coordinate ring of $\PP^b$, $M_S$ for the full subcategory of the category of graded $S$-modules whose objects are finite direct sums of $S(-i)$ where $0 \le i \le b$, $K_S$ for the homotopy category of bounded complexes over $M_S$, $D^{\mathbf{b}}_{\mathrm{Coh}}(\PP^b)$ for the bounded derived category of coherent sheaves on $\PP^b$. By \cite{beilinson}, there is an equivalence of categories $F \colon K_S \lra D^{\mathbf{b}}_{\mathrm{Coh}}(\PP^b)$ such that $F$ sends a complex of $S$-modules to its sheafification. The claim now follows from 
\[
F(H^0_\ast(\mc{S}_{b-1}^{\bullet})) = \mc{S}_{b-1}^{\bullet} \simeq \mc{W}_{b-1}^{\bullet} = F(H^0_\ast(\mc{W}_{b-1}^{\bullet})).
\]

Finally, since $H^0_\ast(\mc{S}_{b-1}^{\bullet})$ and $H^0_\ast(\mc{W}_{b-1}^{\bullet})$ are minimal complexes, they must be isomorphic, and therefore $\mc{S}_{b-1}^{\bullet}\cong \mc{W}_{b-1}^{\bullet}$ are in fact isomorphic as complexes. 
\end{proof}

We establish the following important relation between different complexes $\mc{S}_r^{\bullet}$.

\begin{proposition}\label{prop:iso-Sr-Sr-1}
 If $d=ab$ as above, then we have quasi-isomorphisms
\begin{equation*}
 \mc{S}_r^{\bullet} \simeq \mc{S}_{r-1}^{\bullet}(a)\text{ for all }r\geq b.
\end{equation*}
\end{proposition}

\addtocounter{equation}{-1}
\begin{subequations}
\begin{proof}
 To simplify the notation, we write $\mc{C}^{-1} = \Sym^{d+b-2}U(-1)$, $\mc{C}^0 = \Sym^d U \oo \mc{O}_{\PP^b}$, and with $\mc{F}$ as in Section~\ref{subsec:presen-Oa}, we let
 \[ \mc{M} = \ker(\mc{C}^0 \onto \mc{O}_{\PP^b}(a)) = \coker(\mc{F}\hookrightarrow\mc{C}^{-1}).\]
 The short exact sequence
 \[ 0 \lra \mc{M} \lra \mc{C}^0 \lra \mc{O}_{\PP^b}(a) \lra 0\]
 induces for each $i\geq 0$ a short exact sequence
 \[ 0 \lra \Sym^i\mc{M} \lra \Sym^i\mc{C}^0 \lra \left(\Sym^{i-1}\mc{C}^0\right)(a) \lra 0.\]
 Tensoring the above sequence with $\Sym^{r-i}\mc{C}^{-1}$, we can assemble the resulting sequences into a short exact sequence of complexes
 \begin{equation}\label{eq:ses-Kr-Sr-Sr-1}
  0 \lra \mc{K}_r^{\bullet} \lra \mc{S}_r^{\bullet} \lra \mc{S}_{r-1}^{\bullet}(a) \lra 0,
 \end{equation}
 where $\mc{K}_r^{\bullet}$ is the (right) Koszul resolution of $\bw^r\mc{F}$ 
  \[ \left[0 \lra \bw^r\mc{F} \lra \right] \bw^r\mc{C}^{-1} \lra \bw^{r-1}\mc{C}^{-1}\oo\mc{M} \lra \cdots \lra \Sym^r\mc{M} \lra 0,\]
 induced by the short exact sequence
 \[ 0 \lra \mc{F} \lra \mc{C}^{-1} \lra \mc{M} \lra 0.\]
 Since $\rk(\mc{F})=b-1$, it follows that for $r\geq b$ we have $\bw^r\mc{F}=0$, so $\mc{K}_r^{\bullet}$ is exact and therefore the short exact sequence (\ref{eq:ses-Kr-Sr-Sr-1}) induces the desired quasi-isomorphism.
\end{proof}
\end{subequations}

For the next result we need to consider the following natural extensions of complexes $\mc{A}^{\bullet}$ concentrated in non-positive degrees ($\mc{A}^i=0$ for $i>0$). For such complexes we have
\[ \mc{H}^0(\mc{A}^{\bullet}) = \coker(\mc{A}^{-1}\lra\mc{A}^0).\]
We define the extended complex $\widehat{\mc{A}}^{\bullet}$ by letting 
\[\widehat{\mc{A}}^1 = \mc{H}^0(\mc{A}^{\bullet}),\text{ and }\widehat{\mc{A}}^i = \mc{A}^i\text{ for }i\neq 1,\]
with differential $\pd^0\colon \widehat{\mc{A}}^0 \lra \widehat{\mc{A}}^1$ given by the natural surjection $\mc{A}^0\onto\mc{H}^0(\mc{A}^{\bullet})$. Notice that by construction, $\mc{H}^i(\widehat{\mc{A}}^{\bullet}) = 0$ for $i\geq 0$, so the cohomology of $\widehat{\mc{A}}^{\bullet}$ is concentrated in strictly negative degrees. Moreover, if $\mc{A}^{\bullet}\lra\mc{B}^{\bullet}$ is a quasi-isomorphism, then the induced map $\widehat{\mc{A}}^{\bullet}\lra\widehat{\mc{B}}^{\bullet}$ is also a quasi-isomorphism.

We apply this construction to the complexes $\mc{W}_r^{\bullet}$. Since $\mc{H}^0(\mc{W}_1^{\bullet}) = \coker(\varphi^{\vee}(-1))=\mc{F}^{\vee}(-1)$, we get that
 \[\mc{H}^0(\mc{W}_r^{\bullet}) = \coker(\mc{W}_r^{-1}\lra\mc{W}_r^0) = \left(\bw^r\mc{F}^{\vee}\right)(-r)\text{ for all }r\geq 1.\]
 Since $\rk(\mc{F}^{\vee})=b-1$, this shows that $\widehat{\mc{W}}_r^{\bullet}=\mc{W}_r^{\bullet}$ for $r\geq b$. Moreover, when $r=b-1$ we have
 \[\widehat{\mc{W}}_{b-1}^1=\mc{H}^0(\mc{W}_{b-1}^{\bullet})=\mc{O}_{\PP^b}(a(b-1)).\]
 
\begin{proposition}\label{prop:iso-Wr-Wr-1}
 If $d=ab$ as above, then we have a quasi-isomorphism
\begin{equation*}\label{eq:Wr-Wr-1-a}
 \mc{W}_{b-2}^{\bullet}(-a-1)[1]\simeq  \widehat{\mc{W}}_{b-1}^{\bullet}.
\end{equation*}
\end{proposition}

\begin{proof}
 In analogy with the proof of Proposition~\ref{prop:iso-Sr-Sr-1}, we let 
 \[\mc{D}^{-1} = \Sym^d U(-1),\quad \mc{D}^0 = \Sym^{d+b-2}U \oo \mc{O}_{\PP^b},\] 
 and note that $\varphi^{\vee}\colon \mc{D}^{-1}\lra\mc{D}^0$ has kernel $\mc{O}_{\PP^b}(-a-1)$ and cokernel $\mc{F}^{\vee}(-1)$. We let
  \[ \mc{N} = \ker(\mc{D}^0 \onto \mc{F}^{\vee}(-1)) = \coker(\mc{O}_{\PP^b}(-a-1)\hookrightarrow\mc{D}^{-1}).\]
The short exact sequence
\[ 0 \lra \mc{O}_{\PP^b}(-a-1) \lra \mc{D}^{-1} \lra \mc{N} \lra 0\]
induces for each $i\geq 0$ a short exact sequence
\[ 0 \lra \left(\Sym^{i-1}\mc{D}^{-1}\right) (-a-1) \lra \Sym^i\mc{D}^{-1} \lra \Sym^i \mc{N} \lra 0.\]
Tensoring the above sequence with $\bw^{b-1-i}\mc{D}^0$ and varying $i$, we get a short exact sequence of complexes
\[ 0 \lra \mc{W}_{b-2}^{\bullet}(-a-1)[1] \lra \mc{W}_{b-1}^{\bullet} \lra \mc{G}^{\bullet} \lra 0,\]
where $\mc{G}^{\bullet}$ is the (left) Koszul resolution of $\bw^{b-1}(\mc{F}^{\vee}(-1))=\mc{O}_{\PP^b}(a(b-1))$ induced by the short exact sequence
\[ 0 \lra \mc{N} \lra \mc{D}^0 \lra \mc{F}^{\vee}(-1) \lra 0.\]
The map $\mc{W}_{b-1}^{\bullet} \lra \mc{G}^{\bullet}$ induces an isomorphism $\mc{H}^0(\mc{W}_{b-1}^{\bullet}) \cong \mc{H}^0(\mc{G}^{\bullet})$, hence we get a short exact sequence of complexes
\[ 0 \lra \mc{W}_{b-2}^{\bullet}(-a-1)[1] \lra \widehat{\mc{W}}_{b-1}^{\bullet} \lra \widehat{\mc{G}}^{\bullet} \lra 0.\]
Since $\widehat{\mc{G}}^{\bullet}$ is exact, this induces the desired quasi-isomorphism, which concludes the proof.
\end{proof}

\subsection{Character computation} \label{sec:char-comp}

By comparing term by term, the isomorphism of complexes in Theorem~\ref{thm:self-duality} implies that for all $i$, we have an isomorphism of $\GL(U)$-representations
\[
\bigwedge^i(\Sym^{d+b-2} U) \otimes \Sym^{b-1-i}(\Sym^d U) \cong \Sym^i(\Sym^d U) \otimes \bigwedge^{b-1-i}(\Sym^{d+b-2} U),
\]
and similar isomorphisms were obtained in relation to the classical Hermite reciprocity in \cite{rai-sam}*{Corollary~5.4}. We can directly show that these are isomorphic by comparing their characters. The character of $\Sym^n(U)$ is $\sum_{i=0}^n x_1^i x_2^{n-i}$, and we can specialize $x_1 \mapsto q$ and $x_2 \mapsto 1$ without losing information (as long we remember the homogeneous degree $n$) to get 
\[
1+q+\cdots + q^{n} =: [n+1]_q.
\]
Next, define the $q$-factorial by 
\[
[n]_q! := [n]_q [n-1]_q \cdots [1]_q,
\]
and the $q$-binomial coefficient by
\[
\begin{bmatrix} n \\ m \end{bmatrix}_q := \frac{[n]_q!}{[m]_q! [n-m]_q!}.
\]
Analogously, we can define multinomial coefficients. A standard computation gives
\[
\mathrm{char} (\bigwedge^m(\Sym^n U)) = \begin{bmatrix} n+1 \\ m \end{bmatrix}_q, \qquad
\mathrm{char} (\Sym^m(\Sym^n U)) = \begin{bmatrix} m+n \\ m\end{bmatrix}_q.
\]
In particular, we have
\begin{align*}
\mathrm{char} (\bigwedge^i(\Sym^{d+b-2} U) \otimes \Sym^{b-1-i}(\Sym^d U)) &= \frac{[d+b-1]_q!}{[i]_q! [d+b-1-i]_q!} \cdot \frac{ [d+b-1-i]_q!} {[b-1-i]_q! [d]_q!}\\
&= \begin{bmatrix} d+b-1 \\ i, b-1-i, d \end{bmatrix}_q\\
&=\frac{[d+i]_q!}{[i]_q! [d]_q!} \cdot \frac{ [d+b-1]_q!} {[b-1-i]_q! [d+i]_q!}\\
&= \mathrm{char} (\Sym^i(\Sym^d U) \otimes \bigwedge^{b-1-i}(\Sym^{d+b-2} U)).
\end{align*}
This suggests that the isomorphism of complexes in Theorem~\ref{thm:self-duality} can be extended to the case that $b$ does not divide $d$. Furthermore, it also suggests the existence of a larger family of isomorphisms (of complexes) whose characters are products of three or more $q$-binomial coefficients.

\section{Equations and syzygies for $X_{(a^b)}$}\label{sec:Xab}

The goal of this section is to determine the equations and syzygies of the variety $X=X_{(a^b)}$. A key ingredient for this is the maximal rank property of the Foulkes--Howe maps \eqref{eq:def-alfak}. We let $d=ab$ as usual, and observe that we can realize $X$ as a linear projection of a Veronese variety
\[
\xymatrix{
& \PP^N \ar@{-->}[dr] & \\
\PP^b \ar@{^{(}->}[ur]^{|\mc{O}_{\PP^b}(a)|} \ar[rr]_{\nu_a} & & \PP^d\\
}
\]
where $\PP^b=\bb{P}(\Sym^b U)$, $\PP^d=\bb{P}(\Sym^d U)$, $N={a+b\choose a}-1$, $\PP^N = \bb{P}(\Sym^a(\Sym^b U))$, and the linear projection is induced by the inclusion
\[ \Sym^d U = \Sym^{ab}U \hookrightarrow \Sym^a(\Sym^b U).\]
We summarize some notation, and recall some basic facts from elimination theory:
\begin{itemize}
 \item $S = \Sym(\Sym^d U)$ is the homogeneous coordinate ring of $\PP^d$.
 \item $R = \Sym(\Sym^b U)$ is the homogeneous coordinate ring of $\PP^b$.
 \item $A = \Sym(\Sym^a(\Sym^b U))$ is the homogeneous coordinate ring of $\PP^N$.
 \item $\a\colon S\lra R$ is the ring homomorphism corresponding to $\nu_a$. 
 \item $R^{(a)} $ is the homogeneous coordinate ring of the degree $a$ Veronese embedding of $\PP^b$.
 \item $B = \op{Im}(\a)$ is the coordinate ring of $X$, and $I(X) = \ker(\alpha)$ is the ideal of $X$.
\end{itemize}

\subsection{The Foulkes--Howe map has maximal rank} \label{subsec:FH-map}

We note that the maps \eqref{eq:def-alfak} are obtained from the algebra homomorphism $\a\colon S\lra R$ by restricting to degree $k$. We now use the results of Section~\ref{sec:self-duality} to prove that $\a_k$ has maximal rank.

\begin{proof}[Proof of Theorem~\ref{thm:Foulkes-Howe-map}]
  Recall the construction and properties of the complexes $\mc{S}_r^{\bullet}$ from Section~\ref{sec:self-duality}. Since $\mc{H}^0(\mc{S}_1^{\bullet}) = \coker(\varphi)=\mc{O}_{\PP^b}(a)$, it follows that
 \[\mc{H}^0(\mc{S}_r^{\bullet}) = \coker(\mc{S}_r^{-1}\lra\mc{S}_r^0) = \mc{O}_{\PP^b}(ra)\text{ for all }r\geq 1.\]
 If we consider the extended complexes $\widehat{\mc{S}}_r^{\bullet}$ as in Section~\ref{sec:self-duality}, then we have $\widehat{\mc{S}}_r^1=\mc{O}_{\PP^b}(ra)$ and
 \[ \a_r = H^0(\PP^b,\pd^0_r)\]
 is the map induced on global sections by $\pd^0_r\colon \widehat{\mc{S}}_r^0\lra \widehat{\mc{S}}_r^1$.
 
 We first consider the borderline case and prove that the map $\a_b$ is an isomorphism. Since $\a_b$ is the degree $b$ component of an algebra homomorphism whose source is a polynomial ring, this then shows that $\a_k$ is injective for all $k\leq b$. 
 
 Since the sheaves $\widehat{\mc{S}}_b^i$ have no cohomology for $i\neq 0,1$, and $\widehat{\mc{S}}_b^0$, $\widehat{\mc{S}}_b^1$ have no higher cohomology, the hypercohomology spectral sequence applied to $\widehat{\mc{S}}_b^{\bullet}$ shows that
 \[ \bb{H}^{0}(\widehat{\mc{S}}_b^{\bullet}) = \ker(\a_b),\ \bb{H}^{1}(\widehat{\mc{S}}_b^{\bullet}) = \coker(\a_b),\text{ and }\bb{H}^i(\widehat{\mc{S}}_b^{\bullet})=0\text{ for }i\neq 0,1.\]
To prove that $\a_b$ is an isomorphism is thus equivalent to showing that the hypercohomology of $\widehat{\mc{S}}_b^{\bullet}$ is identically zero. Note that Proposition~\ref{prop:iso-Sr-Sr-1} induces quasi-isomorphisms between the extended complexes
 \[ \widehat{\mc{S}}_r^{\bullet} \simeq \widehat{\mc{S}}_{r-1}^{\bullet}(a)\text{ for all }r\geq b.\]
 Combining this with Theorem~\ref{thm:self-duality} and Proposition~\ref{prop:iso-Wr-Wr-1}, we get
 \[ \widehat{\mc{S}}_b^{\bullet} \simeq \widehat{\mc{S}}_{b-1}^{\bullet}(a) \cong \widehat{\mc{W}}_{b-1}^{\bullet}(a) \simeq \mc{W}_{b-2}^{\bullet}(-1)[1].\]
 Since the sheaves forming the complex $\mc{W}_{b-2}^{\bullet}(-1)$ have vanishing cohomology groups, the hypercohomology of $\mc{W}_{b-2}^{\bullet}(-1)$ is identically zero, so the same is true for $\widehat{\mc{S}}_b^{\bullet}$, as desired.
 
 To conclude, we have to check that $\a_k$ is surjective for $k>b$, or equivalently that
 \[ \bb{H}^i(\widehat{\mc{S}}_k^{\bullet}) = 0\text{ for }i\geq 1.\]
 We argue as before to obtain
 \[\widehat{\mc{S}}_k^{\bullet} \simeq \widehat{\mc{S}}_{k-1}^{\bullet}(a) \simeq \cdots \simeq \widehat{\mc{S}}_{b-1}^{\bullet}((k-b+1)a) \cong \widehat{\mc{W}}_{b-1}^{\bullet}((k-b+1)a) \simeq \mc{W}_{b-2}^{\bullet}((k-b)a-1)[1].\]
 Since $\mc{W}_{b-2}^{\bullet}((k-b)a-1)[1]$ is a complex concentrated in negative cohomological degrees, whose terms are sheaves with vanishing higher cohomology, it follows that in fact
 \begin{equation}\label{eq:Hi-hat-Sk}
 \bb{H}^i(\widehat{\mc{S}}_k^{\bullet}) = \bb{H}^i(\mc{W}_{b-2}^{\bullet}((k-b)a-1)[1]) = 0\text{ for }i\geq 0,
 \end{equation}
 concluding the proof.
\end{proof}

\subsection{The ideal of $X$ has a linear resolution}
\label{subsec:lin-res-IX}

We have that $B\subseteq R^{(a)}$ is a graded subring, where both $B,R^{(a)}$ are endowed with the standard grading coming from their realization as quotients of the polynomial rings $S$ and $A$ respectively. More precisely, we have that $B_k = \op{Im}(\a_k)$, and $I(X)_k=\ker(\a_k)$. This leads to the first part of Theorem~\ref{thm:Xab}, as follows.

\begin{theorem}\label{thm:I-Xab}
 The ideal $I(X)$ is generated in degree $(b+1)$ and has a linear resolution.
\end{theorem}

\begin{proof}
 By Theorem~\ref{thm:Foulkes-Howe-map}, if $k\leq b$ then $\a_k$ is injective, hence $I(X)_k=0$. The minimal generators of $I(X)$ occur therefore in degree $(b+1)$ or higher. To conclude the proof of the theorem, it is then enough to verify that the Castelnuovo--Mumford regularity of $B=S/I(X)$ is less than or equal to $b$. By Theorem~\ref{thm:Foulkes-Howe-map}, $\a_k$ is surjective for $k\geq b$, so $B_k = R^{(a)}_k$ in this range. It follows that
 \[ 
 \reg(B) \le \max \{\reg(B_{\geq b}), b\} = \max\{\reg (R^{(a)}_{\geq b} ), b\} = \max\{\reg (R^{(a)}),b\}.
 \]
 Finally, using \cite{cox-materov}*{Theorem~1.4}, we get that 
 \[ 
 \reg\left(R^{(a)}\right)= b-\left\lfloor\frac{b}{a}\right\rfloor \leq b. \qedhere
 \]
\end{proof}

\subsection{The ideal of $X$ is determinantal}
\label{subsec:IX-determinantal}

We next verify the assertion in Theorem~\ref{thm:Xab} regarding the determinantal structure for $I(X)$. As noted in the Introduction, the $(b+1)$-minors of \eqref{eq:def-omega-onPd} vanish on $X$. We can view the space spanned by these minors as the image of a natural map
\[ \beta\colon  \bw^{b+1}(\Sym^{d+b-2}U) \cong \bw^{b+1}(\Sym^b U) \oo \bw^{b+1}(\Sym^{d+b-2}U) \lra I(X)_{b+1} \subseteq S.\]
By Theorem~\ref{thm:I-Xab}, $I(X)$ is generated in degree $(b+1)$, so we need to check that $\beta$ is surjective.

\begin{lemma}\label{lem:beta-surj}
 We have an exact sequence
 \[ \bw^{b+1}(\Sym^{d+b-2}U) \overset{\beta}{\lra} \Sym^{b+1}(\Sym^{d}U) \overset{\a_{b+1}}{\lra} \Sym^{a(b+1)}(\Sym^b U) \lra 0.\]
 In particular, since $I(X)_{b+1} = \ker(\a_{b+1})$, the image of $\beta$ generates the ideal $I(X)$.
\end{lemma}

\begin{proof} We consider the extended complex $\widehat{\mc{S}}_{b+1}^{\bullet}$ as in Section~\ref{subsec:FH-map}, noting that $\mc{H}^i(\widehat{\mc{S}}_{b+1}^{\bullet}) = 0$ for $i=0,1$, and that the cohomology groups of the sheaves $\widehat{\mc{S}}_{b+1}^i$ vanish identically for $-b\leq i\leq -1$. Moreover, the only non-zero cohomology group of the sheaf $\widehat{\mc{S}}_{b+1}^{-b-1}$ is
\[ H^b(\PP^b,\widehat{\mc{S}}_{b+1}^{-b-1}) = \bw^{b+1}(\Sym^{d+b-2}U),\]
while $\widehat{\mc{S}}_{b+1}^{0}$, $\widehat{\mc{S}}_{b+1}^{1}$ have non-zero cohomology
\[ H^0(\PP^b,\widehat{\mc{S}}_{b+1}^0) = \Sym^{b+1}(\Sym^{d}U),\quad H^0(\PP^b,\widehat{\mc{S}}_{b+1}^1) = \Sym^{a(b+1)}(\Sym^b U).\]
It follows using the discussion in Section~\ref{subsec:max-minors} that the hypercohomology of $\widehat{\mc{S}}_{b+1}^{\bullet}$ is the same as the cohomology of the complex
\[ 0 \lra \bw^{b+1}(\Sym^{d+b-2}U) \overset{\beta}{\lra} \Sym^{b+1}(\Sym^{d}U) \overset{\psi_{b+1}}{\lra} \Sym^{a(b+1)}(\Sym^b U) \lra 0,\]
where the middle term is in cohomological degree $0$. To conclude, it is then enough to check that $\bb{H}^i(\widehat{\mc{S}}_{b+1}^{\bullet})=0$ for $i\geq 0$, which is the special case $k=b+1$ of (\ref{eq:Hi-hat-Sk}).
\end{proof}

\subsection{The algebraic Betti numbers of $X$}
\label{subsec:Betti-X}

We compute the Betti numbers of $X$ explicitly, which in particular verifies the last assertion in Theorem~\ref{thm:Xab} regarding the projective dimension of $I(X)$. The coordinate ring $B=S/I(X)$ has positive depth, hence its minimal free $S$-resolution takes the form
\[ 0 \lla B \lla S \lla F_1 \lla F_2\lla \cdots \lla F_d \lla 0.\]
Moreover, by Theorem~\ref{thm:I-Xab} we can write $F_i = S(-b-i)^{\beta_i}$, where
\begin{equation*}\label{eq:bet-i}
 \beta_i = \dim_{\bb{C}}\Tor_i^S(B,\CC)_{b+i}\quad\text{ for }i=1,\dots,d.
\end{equation*}

\begin{proposition}\label{prop:explicit-betti}
    If $a,b\geq 2$ then the Betti numbers are given by
    \[\beta_i = \sum_{j = 0}^{i-1} (-1)^{i+j-1} \binom{d+1}{j}\left(\binom{d+b+i-j}{d} - \binom{d+a(i-j)+b}{b} \right).\]
    In particular,
    \[\beta_{1} = \binom{d+b+1}{b+1} - \binom{d+a+b}{b}\quad\text{and}\quad \beta_{d} = \binom{d-a+b}{b}-\binom{d+b-1}{b-1},\]
    so $B$ has projective dimension $d$, and $I(X)$ has projective dimension $d-1$.
\end{proposition}

\begin{proof} By Theorem~\ref{thm:Foulkes-Howe-map}, the Hilbert function of $I(X)$ is given in degrees $k\geq b+1$ by
\[
\mathrm{HF}_{I(X)}(k) = \dim_{\CC}\Sym^k(\Sym^d U) -  \dim_{\CC}\Sym^{ak}(\Sym^b U) = \binom{d+k}{d}-\binom{ak+b}{b}.
\]
The Hilbert series of $I(X)$ is then $\mathrm{HS}_{I(X)}(z) = z^{b+1}\cdot G(z)$, where
\[G(z) = \sum_{k\geq 0} \mathrm{HF}_{I(X)}(b+1+k) \cdot z^k = \sum_{k\geq 0} \left(\binom{d+b+1+k}{d} - \binom{d+a(1+k)+b}{b} \right)\cdot z^k.\]
Using the resolution $F_{\bullet \geq 1}$ of $I(X)$, we can also write 
\[\mathrm{HS}_{I(X)}(z) = \frac{\b_1z^{b+1}-\b_2 z^{b+2}+\cdots+(-1)^{d-1}\b_d z^{b+d}}{(1-z)^{d+1}}.\] 
It follows that
\[\sum_{i=1}^d (-1)^{i-1}\b_i z^{i-1} = (1-z)^{d+1}\cdot G(z)\]
and the desired formula for $\b_i$ follows by equating the coefficients of $z^{i-1}$. The formula for $\b_1$ follows immediately, while that for $\b_d$ requires a binomial identity.

We compute $\b_d$ in a different way, by considering the short exact sequence
\[ 0 \lra B \lra R^{(a)} \lra M \lra 0\]
where $M$ is a module of finite length. It follows from \cite{DE-syzygies}*{Corollary 4.4} and Theorem~\ref{thm:Foulkes-Howe-map} that
\begin{equation}\label{eq:regM-flen}
\op{reg}(M) = \max\{i : M_i\neq 0\} = b-1.
\end{equation}
Since $R^{(a)}$ is Cohen--Macaulay of dimension $b+1$, it follows that $\Tor^S_i(R^{(a)},\CC)=0$ for $i>d-b$. The long exact sequence for $\Tor$ groups then yields
\[ \Tor^S_{d+1}(M,\CC)_{d+b} \cong \Tor^S_d(B,\CC)_{d+b},\]
and the left side identifies via Koszul homology with the degree $(b-1)$ elements of $M$ annihilated by the maximal ideal. Since $M_b=0$, we get $(0:_M \mf{m})_{b-1}=M_{b-1}$ which implies that
\[ \b_d = \mathrm{HF}_M(b-1) = \binom{d-a+b}{b}-\binom{d+b-1}{b-1}.\]
Since $\b_d>0$ when $a,b\geq 2$, this proves that $B$ has projective dimension $d$, and $I(X)$ has projective dimension $d-1$, concluding our proof.
\end{proof}

\subsection{Examples}\label{sec:differentHermite} We conclude by showing that $\alpha_b$ is not a scalar multiple of the classical Hermite isomorphism in general. Consider the case $\lambda = (2,2)$, and recall the map 
\[\alpha_2 \colon \Sym^2(\Sym^4 U) \lra \Sym^4(\Sym^2 U)\]
from (\ref{eq:weird-Hermite}). Following the notation of \cite{landsberg}*{Chapter~9}, we denote the corresponding classical Hermite isomorphism by $h_{2,4}$. As $\SL(U)$-representations, both $\Sym^2(\Sym^4 U)$ and $\Sym^4(\Sym^2 U)$ decompose as 
\[\Sym^8 U \oplus \Sym^4 U \oplus \Sym^0 U.\] 
We let $z_i = x_1^{4-i}x_2^{i}$ for $0\le i \le 4$ and let $y_j = x_1^{2-j}x_2^j$ for $0 \le j \le 2$, so that $\Sym^4 U$ has basis $\{z_0,z_1,\dots,z_4\}$ and $\Sym^2 U$ has basis $\{y_0,y_1,y_2\}$. The corresponding highest weight vectors, $v_8,v_4,v_0$ for $\Sym^2(\Sym^4 U)$, and $w_8,w_4,w_0$ for $\Sym^4(\Sym^2 U)$, are given by
\begin{align*}
    &v_8 = z_0^2, && v_4 = z_0z_2 - z_1^2, && v_0 = z_0z_4 - 4z_1z_3 + 3z_2^2,\\
    &w_8 = y_0^4, && w_4 = y_0^2y_1^2 - y_0^3y_2, && w_0 = y_1^4 - 2 y_0y_1^2 y_2 + y_0^2 y_2^2.
\end{align*}
By Schur's lemma, any $\SL(U)$-equivariant map $\theta \colon \Sym^2(\Sym^4 U) \lra \Sym^4(\Sym^2 U)$ is uniquely determined by the triple $p(\theta) = (a_8,a_4,a_0) \in \CC^3$ where $\theta(v_i) = a_i w_i$ for $i = 8,4,0$. A direct computation gives $p(\alpha_2) = (1, -\frac{1}{3}, \frac{4}{3})$ and $p(h_{2,4}) = (1,\frac{1}{4}, \frac{1}{2})$. In particular, $\alpha_2$ and $h_{2,4}$ are non-proportional isomorphisms.

\section{Regularity and syzygies of some primary ideals}
\label{sec:reg-Iab}

This section collects some results regarding the Castelnuovo--Mumford regularity of the ideals $I_{a,b}$, as well as the minimal free resolutions of $I_{a,b}^{b-1}$ and $I_{a,3}$. We write $S=\Sym(\Sym^b U)$ for the homogeneous coordinate ring of $\PP^b$, the ambient ring for the ideal $I_{a,b}$.

\subsection{The regularity of $I_{a,b}$}
\label{subsec:reg-Iab}

In this section we derive Theorem~\ref{thm:reg-Iab} from the results in \cite{weyman-mult}. Using \eqref{eq:regM-flen}, if $R=S/I$ is an Artinian graded algebra then 
\[
\op{reg}(R) = \max\{i : R_i\neq 0\},
\]
and moreover we have $\op{reg}(I) = 1 + \op{reg}(R)$.

\begin{proof}[Proof of Theorem~\ref{thm:reg-Iab}]
We write $R=S/I_{a,b}$, set $c=\left\lfloor\frac{b+2}{2}\right\rfloor$, and let $u_0,\dots,u_b$ denote the standard basis of $\Sym^b U$. By \cite{weyman-mult}*{Theorem 2.4}, the minimal generators of $I_{a,b}$ form a Gr\"obner basis, and the initial ideal $J_{a,b}$ is generated by the monomials 
\[
\{u_i^a \mid i=0,\dots,b\} \cup \{u_i^j u_{i+1}^{a-j} \mid i=0,\dots,b-1,\ 1 \le j \le a-1\}.
\]
Let $u_0^{d_1} \cdots u_b^{d_b}$ be a nonzero monomial in the quotient ring $S/ J_{a,b}$. If $b$ is odd, then we have
\[
d_0 + \cdots + d_b = (d_0+d_1) + \cdots + (d_{b-1}+d_b) \le \frac{b+1}{2}(a-1) = c(a-1).
\]  
Moreover, the monomial $u_0^{a-1} u_2^{a-1} u_4^{a-1} \cdots u_{b-1}^{a-1}$ has degree exactly $c(a-1)$, and it is non-zero in $S/J_{a,b}$. If instead $b$ is even, then we have
\[
d_0 + \cdots + d_b = (d_0+d_1) + \cdots + (d_{b-2}+d_{b-1}) + d_b \le \frac{b}{2}(a-1) + (a-1) = c(a-1).
\]
Moreover, the monomial $u_0^{a-1} u_2^{a-1} u_4^{a-1} \cdots u_b^{a-1}$ has degree exactly $c(a-1)$, and it is non-zero in $S/J_{a,b}$. In both cases, we conclude that $(S/J_{a,b})_d=0$ if $d> c(a-1)$ and $(S/J_{a,b})_{c(a-1)} \ne 0$. Since the Hilbert functions of $R$ and $S/J_{a,b}$ agree, this concludes our proof.
\end{proof}

\begin{remark}
If $b=3$, then Theorem~\ref{thm:reg-Iab} implies that $R_i=0$ whenever $i \ge 2a-1$. This implies that the composition (where the second map is multiplication)
\[
\Sym^{a-1}(\Sym^3 \CC^2) \otimes \Sym^{3a}(\CC^2) \to \Sym^{a-1}(\Sym^3 \CC^2) \otimes \Sym^{a}(\Sym^3 \CC^2) \to \Sym^{2a-1}(\Sym^3 \CC^2)
\]
is surjective, which recovers \cite{faenzi}*{Proposition 3.11}. We note that the proof in loc. cit. appeals to Hermite reciprocity, but does not make precise which Hermite reciprocity isomorphism is being used (see also the discussion regarding \eqref{eq:weird-Hermite}, and Section~\ref{sec:self-duality}).
\end{remark}

\subsection{The resolution of $I_{a,b}^{b-1}$}
\label{subsec:res-Iab-b-1}

Conjecture~\ref{conj:reg-pow-conj} implies that $I_{a,b}^{b-1}$ has Castelnuovo--Mumford regularity $ab-1$. We confirm this by providing the minimal free resolution of $I_{a,b}^{b-1}$. If $a=1$ then $I_{a,b}=\mf{m}$ is the maximal homogeneous ideal and $I_{a,b}^{b-1}$ has a well-known Eagon--Northcott resolution, so we will consider the non-trivial case $a\geq 2$.

\begin{theorem}\label{thm:res-Iab-b-1}
Suppose that $a\geq 2$, $d=ab$, and consider the Betti numbers of $I_{a,b}^{b-1}$,
\[ \beta_{i,j} = \dim_{\bb{C}}\Tor_i^S\left(I_{a,b}^{b-1},\CC\right)_{i+j}.\]
The non-vanishing Betti numbers are concentrated in two linear strands, and are given by
    \begin{align*}
        \beta_{i,d-a} &= \binom{d+b-1-i}{d}\cdot \binom{d+b-1}{i}\quad \text{for } i=0,\dots,b-1,\\
        \beta_{i,d-1} &=\binom{d + b -1}{b-i} \cdot \binom{d+i-2}{d}\quad \text{for }i=2,\dots,b.
    \end{align*}
    In particular, the Castelnuovo--Mumford regularity of $I_{a,b}^{b-1}$ is $\reg(I_{a,b}^{b-1}) = d-1$.
\end{theorem}

\begin{proof} Combining Proposition~\ref{prop:iso-Wr-Wr-1} with Theorem~\ref{thm:self-duality} we get a quasi-isomorphism
\[\mc{W}_{b-2}^{\bullet}(-a-1)[1]\overset{\sim}{\lra}  \widehat{\mc{S}}_{b-1}^{\bullet}\]
and we denote by $\mc{C}^{\bullet}$ the associated mapping cone. It is an exact complex with $\mc{C}^i=0$ for $i>1$ and $i<-b$, 
\begin{align*}
\mc{C}^1 &= \mc{O}_{\PP^b}(a(b-1)) = \mc{O}_{\PP^b}(d-a),\\
\mc{C}^0 &= \Sym^{b-1}(\Sym^d U)\oo\mc{O}_{\PP^b}, \\
\mc{C}^{-i} &= \mc{S}_{b-1}^{-i}\oplus\mc{W}_{b-2}^{-i+2}(-a-1)\text{ for }i=1,\dots,b.
\end{align*}
In particular, each $\mc{C}^i$ is a direct sum of line bundles, and the differential $\pd\colon \mc{C}^0\lra\mc{C}^1$ is obtained by taking the $(b-1)$-st power of the linear series $\Sym^d U$ for $\mc{O}_{\PP^b}(a)$. It follows from Lemma~\ref{lem:induced-H0} that the complex
\begin{equation}\label{eq:H0C-res-Iab-b-1}
0 \lra H^0_*\left(\mc{C}^{-b}\right) \lra \cdots \lra H^0_*\left(\mc{C}^{-1}\right) \lra H^0_*\left(\mc{C}^{0}\right)
\end{equation}
gives a resolution of the image of $H^0_*\left(\mc{C}^{0}\right)\lra H^0_*\left(\mc{C}^{1}\right)$, which by construction equals $I_{a,b}^{b-1}$. Keeping track of grading, it follows that the minimal graded free resolution $F_{\bullet}$ of $I_{a,b}^{b-1}$ is given by
\begin{align*}
F_i &= \bw^i\Sym^{d+b-2}U \oo \Sym^{b-1-i}(\Sym^d U) \oo S(-d+a-i) \oplus \\
&\qquad \bw^{b-i}\Sym^{d+b-2}U \oo \Sym^{i-2}(\Sym^d U) \oo S(-d-i+1).
\end{align*}
The formulas for the Betti numbers follow by direct calculation, while the regularity statement follows from \cite{DE-syzygies}*{Corollary 4.5}.
\end{proof}

\subsection{Minimal free resolutions via instanton bundles on binary cubics}
\label{subsec:syzygies-Ia3}

In this section we assume that $b=3$, and our goal is to describe the minimal free resolution for the ideals $I_{a,3}\subset S$. 

\begin{proof}[Proof of Theorem \ref{thm:res-Ia3}]
Using the notation in Section~\ref{subsec:presen-Oa}, we let $\mc{E} = \mc{F}^{\vee}(-a-1)$, so after dualizing (\ref{eq:pres-Oa}) we obtain a resolution
\begin{equation}\label{eq:res-E-faenzi}
 0 \lra \mc{O}_{\PP^3}(-2a-1) \lra A(-a-1) \lra B(-a) \lra \mc{E} \lra 0,
\end{equation}
where $A=\Sym^{3a}U$ and $B=\Sym^{3a+1}U$, and $(-)(i)$ is shorthand for $- \oo \mc{O}_{\PP^3}(i)$. Using \cite{faenzi}*{Theorem~3.8 and Lemma~3.10}, $\mc{E}$ is the unique $\SL_2$-equivariant instanton bundle on $\PP^3$ with second Chern class $c_2(\mc{E}) = {a+1\choose 2}$. By \cite{faenzi}*{Lemma~3.3} (see also \cite{OSS}*{Example~5 in Chapter 2, Section 3.3.2}), it admits a linear monad 
\[
  0 \lra V_{-1}(-1) \lra V_0\oo\mc{O}_{\PP^3} \overset{\pd}{\lra} V_1(1) \lra 0,
\] 
(i.e., the sequence is exact except in the middle, where its cohomology is $\mc{E}$) where
\[
  V_0 = \Sym^{2a+1}U \oplus \Sym^{a-2}U \oo \Sym^{a-1}U,\quad V_1 = \bw^2(\Sym^a U),\quad\mbox{and}\quad V_{-1}=V_1^{\vee}.
\]
If we set $\mc{K}=\ker(\pd)$ then we have a short exact sequence
\[ 0 \lra V_{-1}(-1) \lra \mc{K} \lra \mc{E} \lra 0,\]
and since $H^1(\mc{O}_{\PP^3}(a-1))=0$, it follows from the long exact sequence in cohomology that the surjection $B(-a)\onto\mc{E}$ lifts to a map $B(-a)\lra\mc{K}\subseteq V_0\oo\mc{O}_{\PP^3}$. Since $V_{-1}(-1)$ injects into $\mc{K}$, we get an induced map of complexes
\[
\xymatrix{
0 \ar[r] & \mc{O}_{\PP^3}(-2a-1) \ar[r] &  A(-a-1) \ar[r] \ar@{.>}[d] &  B(-a) \ar[d] & & \\
& 0 \ar[r] & V_{-1}(-1) \ar[r] &  V_0\oo\mc{O}_{\PP^3}  \ar[r] &  V_1(1) \ar[r] & 0\\
}
\]
which by construction is a quasi-isomorphism. The associated mapping cone is then exact, and dualizing, twisting by $\mc{O}_{\PP^3}(-2a-1)$, and identifying $\SL_2$-representations with their duals yields an exact complex 
\[\mc{C}^{\bullet}:\ 0 \lra V_1(-2a-2) \lra V_0(-2a-1)\lra B(-a-1)\oplus V_{-1}(-2a) \lra A(-a) \lra \mc{O}_{\PP^3} \lra 0,\]
where $\mc{O}_{\PP^3}$ is in cohomological degree $1$. We can then apply Lemma \ref{lem:induced-H0} with $b=3$ to get a complex of free modules
\[ F_i = H^0_*\left(\mc{C}^{-i}\right)\]
which is exact except at $F_{-1}$. Since the image of the differential $F_0\lra F_{-1}$ is the ideal $I_{a,3}$, we get a minimal free resolution 
\[ 0 \lla I_{a,3} \lla F_0 \lla F_1 \lla F_2 \lla F_3 \lla 0,
\]
where
\begin{align*} 
F_3 &= V_1\oo S(-2a-2), & F_2 &= V_0\oo S(-2a-1),\\
F_1 &= B\oo S(-a-1)\oplus V_{-1}\oo S(-2a), &  F_0 &= A\oo S(a).
\end{align*}
Computing the dimensions of $A,B,V_{-1},V_0,V_1$ concludes the proof of Theorem~\ref{thm:res-Ia3}.
\end{proof}

\end{document}